\documentclass[10pt,authoryear]{article}
\usepackage[hidelinks,colorlinks=true,linkcolor=blue,citecolor=blue,urlcolor=blue,]{hyperref}
\usepackage[body={15.5cm,21cm}, top=3cm]{geometry}
\usepackage{xcolor}
\usepackage{graphicx}
\usepackage{listings}
\usepackage{indentfirst}
\usepackage{float}
\usepackage{fancyhdr}
\usepackage{amsmath}
\usepackage{caption} 
\usepackage{amssymb} 
\usepackage{amsthm} 
\usepackage{authblk} 

\providecommand{\U}[1]{\protect\rule{.1in}{.1in}}
\providecommand{\U}[1]{\protect \rule{.1in}{.1in}}
\allowdisplaybreaks[1]
\newtheorem{theorem}{Theorem}[section]
\newtheorem{lemma}{Lemma}[section]
\newtheorem{remark}{Remark}[section]
\DeclareMathOperator*{\esssup}{ess\,sup}

\numberwithin{equation}{section} 

\begin{document}
\title{A note on  convergence rate for reflected BSDEs with quadratic generators by penalization method}
\author[a]{Guangyan Jia}
\author[b]{Peng Luo \thanks{\quad
Corresponding author. \\ 
\indent \hspace{1.5em}  E-mail addresses: jiagy@sdu.edu.cn(G. Jia), peng.luo@sjtu.edu.cn(P. Luo), zhumengbo@mail.sdu.edu.cn(M. Zhu)
}}
\author[a]{Mengbo Zhu}

\affil[a]{Zhongtai Securities Institute for Financial Studies, Shandong University, China}
\affil[b]{School of Mathematical Sciences, Shanghai Jiao Tong University, China}
\renewcommand*{\Affilfont}{\small\it}  
\date{}

\maketitle

\begin{abstract}
In this paper, we study the convergence rate between reﬂected backward stochastic diﬀerential equations with quadratic generators and their penalized BSDEs. Using techniques of BMO martingales, we prove the convergence rate is at order $\frac{1}{2}$ as a function of the penalty parameter. Finally, the result is applied to study numerical approximation of reflected BSDEs with sub-quadratic generators by the Euler’s polygonal line method.
\end{abstract}


\section{Introduction}
The theory of nonlinear backward stochastic diﬀerential equations (BSDEs for short) was introduced by Pardoux and Peng \cite{PP}. Subsequently, El Karoui et al. \cite{EKPPQ} considered the reflected BSDEs of the following form
	\begin{equation}
		\begin{cases}			
			Y_t=\xi+\int_t^T f(s,Y_s,Z_s)ds-\int_t^T Z_sdW_s+(K_T-K_t), \\
			Y_t\geq S_t,\ t\in[0,T],\ \textrm{and} \ \int_0^T (Y_s-S_s)dK_s=0,
		\end{cases}
    \nonumber
	\end{equation}
where the generator $f$ is uniformly Lipschitz in $(y,z)$. Compared with the classical BSDEs, the additional non-decreasing process $K$ arises to keep the solution $Y$ above the obstacle $S$. And the Skorohod condition $\int_0^T (Y_s-S_s)dK_s=0$ ensures that the adjustment of $Y$ is minimal. The reflected BSDEs provide a new sight to study optimal stopping problems, which is the foundation of many economic issues, including pricing American claims \cite{EPQ} and the entry decision of a firm into a new market. They also have wide applications in  stochastic control problems \cite{HJ,HZ} and probabilistic interpretation for PDEs \cite{BCFE}.

Due to the significance of BSDEs and reflected BSDEs in theoretical analysis and practical applications, there are a lot of works to weaken the classical assumption. One branch is to relax the Lipschitz condition on generators to quadratic case, i.e., the generators have quadratic growth in the $z$ component. The pioneering work of quadratic BSDEs was done by Kobylanski \cite{K}, where the existence and uniqueness result of quadratic BSDEs with bounded terminal conditions was established. Later, Briand and Hu \cite{BH} studied the quadratic BSDEs with unbounded terminal conditions. The reflected case was explored by Kobylanski et al. \cite{KLQT}  with bounded terminal  conditions and Lepeltier and Xu \cite{LX} with unbounded terminal conditions. For the multi-dimensional quadratic BSDEs and reflected BSDEs, we refer the readers to \cite{CL,HT,LZ24}. The quadratic BSDEs and reflected BSDEs have close connection with exponential utility theory including exponential utility maximization \cite{HIM, M} and stochastic equilibrium \cite{XZ}, pricing contingent claims in incomplete markets \cite{RE}, risk-sensitive control \cite{HT,LZ05,LZ24} and stochastic representations for PDEs \cite{BY,BC}.

In general cases, a large number of BSDEs and reflected BSDEs cannot be solved explicitly. Thus numerical approximation methods are needed. Ma et al. \cite{MPY} proposed the four step algorithm to approximate forward backward stochastic partial differential equations (FBSDEs for short). Bally and Pag{\`e}s \cite{BP} used quantization algorithms for reflected BSDEs whose generators are independent in $z$. Zhang \cite{Z} proved the $L^{2}$-regularity on $Z$, which is helpful for designing a effective numerical algorithm for the decoupled FBSDEs. For FBSDEs with reflection, Ma and Zhang \cite{MZ} attained a semi-explicit representation and some regularity results for $Z$ when the forward process X is an uniformly elliptic diﬀusion. Later, Bouchard and Chassagneux \cite{BbC} extended their result to  the case without the uniform ellipticity condition. In addition, regression-based methods were also prevalent for the numerical approximation of BSDEs and refletced BSDEs (see, e.g. \cite{BT,GL,GLW}). For the quadratic case, Richou \cite{R} and Chassagneux and Richou \cite{CR} investigate the numerical approximation of Markovian BSDEs whose generators have quadratic growth in $z$ and uniform Lipschitz continuity in $y$. Li and Tang \cite{LT} established Euler's polygonal line method for BSDEs with generators of sub-qudratic growth in $z$ and super-linearly growth in $y$. Sun et al. \cite{SLT} carried out the stability and convergence analysis for quadratic reflected BSDEs with unbounded terminal value. 

Recently, some numerical schemes based on penalization method have been established. The penalization method, which frequently appears in the existence problem of the reflected BSDEs, can also be used in numerical field. Specifically, consider the penalized BSDEs of the following form  
\begin{equation} 
\begin{cases}
Y_{t}^{\lambda}=\xi+\int_{t}^{T} f(s,Y_{s}^{\lambda},Z_{s}^{\lambda})ds
 - \int_{t}^{T} Z_{s}^{\lambda} dW_{s} +(K^{\lambda}_T-K^{\lambda}_t),\\
K^{\lambda}_{t}= \lambda \int_{0}^{t} (Y_{s}^{\lambda}-S_{s})^{-}ds.
 \nonumber
\end{cases}
\end{equation}
In the case of the penalty parameter $\lambda$ is large enough, their numerical schemes can be used to approximate the reflected BSDEs. The global error can be divided into two ingredients. One of them is the convergence error between reflected BSDEs and penalized BSDEs, the other is the discretization error of the penalized BSDEs. For this side, M{\'e}min et al. \cite{MPX} proposed a random walk scheme to approximate the penalized BSDEs. The scheme was also used in Martinez et al. \cite{MST} which combines the penalization and the Picard iteration. Bernhart et al. \cite{BPTW} considered the penalization method for  BSDEs with constrained jump. Last, Gobet and Wang \cite{GW} obtained the refined convergence rate between reflected BSDEs and their penalized BSDEs. All of these results are established under the assumption that the generators are Lipschitz continuous in $z$. The new feature of our work lies that we consider the penalization method for  reflected BSDEs with quadratic generators, and obtain the convergence rate with the same order as in \cite{GW}.

In this article, we investigate the convergence rate between penalized BSDEs and reflected BSDEs with quadratic generators. By employing the techniques of BMO martingales and estimates in \cite{GW}, we provide some a priori estimates that are essential to obtain the convergence rate between the reflected BSDEs and penalized BSDEs. The convergence rate is at order 1/2 as a function of penalty parameter $\lambda$. For applications, we studied the penalization method for the numerical scheme of reflected BSDEs with sub-quadratic generators. We used the Euler's polygonal line method propsed in \cite{LT} to approximate the penalized BSDEs with sub-quadratic generators. With the help of the convergence rate above, we get the global discretization error (logarithmic rate), which can be improved to polynomial rate with the additional monotone condition on the generator $f$.

The remainder of the paper is organized as follows. In Section \ref{label:Preliminaries and main Results}, we present some preliminary
results on quadratic BSDEs and quadratc reflected BSDEs. Then, we discuss our main result of convergence rate. In Section \ref{label:Numerical approximation of RBSDEs with sub-quadratic generators}, we establish the numerical approximation for reflected BSDEs with sub-quadratic generators.

\section{Preliminaries and main Results} \label{label:Preliminaries and main Results}
Let $\{W_{t},0\leq t\leq T\}$ be a $d$-dimentional standard Brownian motion definded on a complete probability space $(\Omega,\mathcal{F},\mathit{P})$ and $\{\mathcal{F}_t,0\leq t\leq T\}$ be the augmented natural filtration of $W$. In this paper, equalities and inequalities relevent to random variables and processes are discussed in the $P$-a.s. and $P\otimes dt$-a.e. sense respectively. The Euclidean norm is denoted by $\vert \cdot \vert$. For a random variable $\xi$, its $L^{\infty}$-norm is denoted by $\Vert \xi \Vert_{\infty} :=\esssup_{\omega\in\Omega}\vert\xi(\omega)\vert$.

For $p\geq2$, We denote by 
\begin{itemize}
\item[$\bullet$] $\mathcal{H}^{p}(\mathbb{R}^{n})$  the space of $n$-dimensional predictable processes $Y$ on $[0,T]$ such that\\
\begin{equation}
\Vert Y \Vert_{\mathcal{H}^{p}(\mathbb{R}^{n})}:=E\left[ \left( \int_{0}^{T}\vert Y_{t} \vert^{2}dt \right) ^{\frac{p}{2}} \right] ^{\frac{1}{p}}<\infty;  \nonumber
\end{equation}
and $\mathcal{H}^{p}_{[T_{1},T_{2}]}(\mathbb{R}^{n})$ the space of $n$-dimensional predictable processes $Y$ on $[T_{1},T_{2}]$ such that\\ 
\begin{equation}
\Vert Y \Vert_{\mathcal{H}^{p}_{[T_{1},T_{2}]}(\mathbb{R}^{n})}:=E\left[ \left( \int_{T_{1}}^{T_{2}}\vert Y_{t} \vert^{2}dt \right) ^{\frac{p}{2}} \right] ^{\frac{1}{p}}<\infty;  \nonumber
\end{equation}

\item[$\bullet$] $\mathcal{S}^{p}(\mathbb{R}^{n})$  the space of $n$-dimensional predictable processes $Y$ on $[0,T]$ such that\\
\begin{equation}
\Vert Y \Vert_{\mathcal{S}^{p}(\mathbb{R}^{n})}:=E\left[ \left( \sup\limits_{0\leq t\leq T} \vert Y_{t} \vert^{2} \right) ^{\frac{p}{2}} \right] ^{\frac{1}{p}}<\infty;  \nonumber
\end{equation}

\item[$\bullet$] $\mathcal{S}^{\infty}(\mathbb{R}^{n})$  the space of $n$-dimensional predictable processes $Y$ on $[0,T]$ such that\\
\begin{equation}
\Vert Y \Vert_{\mathcal{S}^{\infty}(\mathbb{R}^{n})}:=\left\Vert  \sup\limits_{0\leq t\leq T} \vert Y_{t} \vert   \right\Vert_{\infty} <\infty; \nonumber
\end{equation}

\item[$\bullet$] $\mathcal{S}^{\infty}_{[T_{1},T_{2}]}(\mathbb{R}^{n})$  the space of $n$-dimensional c{\`a}dl{\`a}g processes $Y$ on $[T_{1},T_{2}]$ such that\\
\begin{equation}
\Vert Y \Vert_{\mathcal{S}^{\infty}_{[T_{1},T_{2}]}(\mathbb{R}^{n})}:=\left\Vert  \sup\limits_{T_{1}\leq t\leq T_{2}} \vert Y_{t} \vert   \right\Vert_{\infty} <\infty. \nonumber
\end{equation}
\end{itemize}
Let $\mathcal{T}$ be the set of all stopping times with values in $[0,T]$. For any uniformly integrable martingale $M$ with $M_{0}=0$, we set
\begin{equation}
\Vert M \Vert_{BMO_{2}}:= \sup\limits_{\tau \in \mathcal{T}} \left\Vert
E \left[ \vert M_{T}-M_{\tau} \vert^{2} \vert \mathcal{F}_{\tau} \right]^{\frac{1}{2}} \right\Vert_{\infty}.  \nonumber
\end{equation}
The class $\{M:\Vert M \Vert_{BMO_{2}} < \infty \}$ is denoted by $BMO$. For
$(\alpha \cdot W)_{T} := \int_{0}^{T} \alpha_{s} dW_{s}$ in $BMO$, the corresponding stochastic exponential is denoted by $\mathcal{E}(\alpha \cdot W)_{0}^{T}$.

Consider a reflected BSDE with generator $f$, terminal value $\xi$ and continuous barrier $S$.\\
\begin{equation} \label{e1}
  \left\{
  \begin{aligned}
  &Y_{t}=\xi+\int_{t}^{T}f(s,Y_{s},Z_{s})ds+K_{T}-K_{t}-\int_{t}^{T}Z_{s}dW_{s}, t\in [0,T], \\
  &Y_{t}\geq S_{t}, 
  \int_{0}^{T}\left(Y_{t}-S_{t}\right)dK_{t}=0    & 
  \end{aligned} 
  \right. 
\end{equation}
where $K$ is a continous and nondecresing process with $K_{0}=0$. For $\lambda>0$, we give the penalized BSDE of \eqref{e1}.   
\begin{equation} \label{e2}
Y_{t}^{\lambda}=\xi+\int_{t}^{T} f(s,Y_{s}^{\lambda},Z_{s}^{\lambda})ds+
\lambda \int_{t}^{T} (Y_{s}^{\lambda}-S_{s})^{-}ds - \int_{t}^{T} 
Z_{s}^{\lambda} dW_{s}
\end{equation}
where we set $K^{\lambda}_{t} = \lambda \int_{0}^{t} (Y_{s}^{\lambda}-S_{s})^{-}ds$.\par
In this paper, we define a generic constant $C>0$, which does not depend on the penalty parameter $\lambda$ and the number of meshes $n$ when discussing the numerical scheme of BSDEs. The constant may change from line to line. We make the following assumptions.\par
$(A1)$ The function $f:\Omega \times [0,T] \times \mathbb{R} \times \mathbb{R}^{d} \rightarrow \mathbb{R}$ satisfies that $f(\cdot,y,z)$ is adapted for each $y \in \mathbb{R}$ and $z \in \mathbb{R}^{d}$. It holds that
\begin{equation}
\begin{aligned}
&\vert f(t,y,z) \vert \leq C(1+ \vert y \vert + \vert z \vert^{2}),\\
&\vert f(t,y_{1},z_{1})-f(t,y_{2},z_{2}) \vert \leq C \vert y_{1}-y_{2} \vert 
+ C(1+ \vert z_{1} \vert + \vert z_{2} \vert)\vert z_{1}-z_{2} \vert,
\end{aligned} \nonumber
\end{equation}
for $y,y_{1},y_{2} \in \mathbb{R}$, $z,z_{1},z_{2} \in \mathbb{R}^{d}$.

$(A2)$ $ S $ is a continuous progressively measurable real-valued process satisfying 
\begin{equation}
\Vert S^{+} \Vert_{\mathcal{S}^{\infty}(\mathbb{R})} \leq C. \nonumber
\end{equation} 

$(A3)$ The terminal condition $\xi \in \mathbb{R}$ is $\mathcal{F}_T$-measurable and satisfies
\begin{equation}
\Vert \xi \Vert_{\infty} \leq C. \nonumber
\end{equation}
In addition, for the purpose of obtaining convergence rate, we need extra assumption for barrier $S$.\par
$(A4)$ The barrier $S$ is a semi-martingale which can be written as 
\begin{equation}
S_{t} = S_{0}+\int_{0}^{t}U_{s}ds+\int_{0}^{t}V_{s}dW_{s}+A_{t} \nonumber
\end{equation}
where $U,V\in \mathcal{S}^{p}$, $A$ is a continuous and non-decreasing process satisfying $A_{T} \in \mathcal{L}^{p}$, $A_{0}=0$.\par
We recall the following results for the existence and uniqueness of solution of reflected BSDE \eqref{e1} and penalized BSDE \eqref{e2} from \cite{CL} and \cite{HT}.

\begin{lemma}\label{l1} Let assumptions $(A1)-(A3)$ be satisfied, then reflected BSDE \eqref{e1} has a unique solution $(Y,Z,K)$ such that $(Y,Z\cdot W,K)\in \mathcal{S}^{\infty}(\mathbb{R})\times BMO \times \mathcal{S}^{p}(\mathbb{R})$ for any $p\geq2$.
\end{lemma}

\begin{lemma}\label{l2} Let assumptions $(A1)-(A3)$ be satisfied, then penalized BSDE \eqref{e2} has a unique solution $(Y^{\lambda},Z^{\lambda})$ such that $(Y^{\lambda},Z^{\lambda}\cdot W) \in \mathcal{S}^{\infty}(\mathbb{R})\times BMO$. Besides, $(Y^{\lambda},Z^{\lambda}\cdot W)$ are uniformly bounded with respect to $\lambda$. That is 
\begin{equation}
\Vert Y^{\lambda} \Vert_{\mathcal{S}^{\infty}(\mathbb{R})} \leq C,\quad  \Vert Z^{\lambda}\cdot W \Vert_{BMO} \leq C .
\end{equation}
\end{lemma}

\begin{theorem}\label{th1} 
Under assumptions $(A1)-(A4)$, let $(Y,Z,K)$ be the solution of reflected BSDE \eqref{e1} and $(Y^{\lambda},Z^{\lambda})$ be the solution of penalized BSDE \eqref{e2}, then it holds that  
\begin{equation} \label{th21}
E \left[ \sup\limits_{0\leq t\leq T} (Y_{t}-Y_{t}^{\lambda})^{p}+\left(\int_{0}^{T} \vert Z_{t}-Z_{t}^{\lambda} \vert^{2} dt \right)^{\frac{p}{2}}+\sup\limits_{0\leq t\leq T}(K_{t}-K_{t}^{\lambda})^{p} \right] \leq \frac{C_{p}}{\lambda^{\frac{p}{2}}}.
\end{equation}
where $C_{p}$ is independent of $\lambda$.
\end{theorem}
\begin{proof}
We divide the proof into three steps:\\
\textbf{Step1} We consider a priori estimation for the process $Y_{t}-Y_{t}^{\lambda}$.
\begin{equation} \label{y-ylambda}
\begin{aligned}
Y_{t}-Y_{t}^{\lambda}&= \int_{t}^{T} (f(s,Y_{s},Z_{s})-f(s,Y_{s}^{\lambda},Z_{s}^{\lambda}))ds+(K_{T}-K_{t})-(K_{T}^{\lambda}-K_{t}^{\lambda}) \\ &\quad- \int_{t}^{T} 
(Z_{s}-Z_{s}^{\lambda}) dW_{s}\\
&=\int_{t}^{T} (f(s,Y_{s},Z_{s})-f(s,Y_{s}^{\lambda},Z_{s})+f(s,Y_{s}^{\lambda},Z_{s})-f(s,Y_{s}^{\lambda},Z_{s}^{\lambda}))ds \\ &\quad+(K_{T}-K_{t})-(K_{T}^{\lambda}-K_{t}^{\lambda})- \int_{t}^{T} 
(Z_{s}-Z_{s}^{\lambda}) dW_{s}\\
&=\int_{t}^{T} (f(s,Y_{s},Z_{s})-f(s,Y_{s}^{\lambda},Z_{s})+\beta_{s}^{\lambda}(Z_{s}-Z_{s}^{\lambda}))ds \\ 
&\quad+(K_{T}-K_{t})-(K_{T}^{\lambda}-K_{t}^{\lambda})- \int_{t}^{T} 
(Z_{s}-Z_{s}^{\lambda}) dW_{s} \\
&=\int_{t}^{T} (f(s,Y_{s},Z_{s})-f(s,Y_{s}^{\lambda},Z_{s}))ds+(K_{T}-K_{t})-(K_{T}^{\lambda}-K_{t}^{\lambda}) \\ 
&\quad- \int_{t}^{T} 
(Z_{s}-Z_{s}^{\lambda}) dW_{s}^{\lambda},
\end{aligned}
\end{equation}
where the process $\beta_{s}^{\lambda}$, according to $(A1)$,  satisfies 
$\vert \beta_{s}^{\lambda} \vert \leq C(1+\vert Z_{s} \vert+\vert Z_{s}^{\lambda} \vert)$. It follows from Lemma \ref{l1} and Lemma \ref{l2} that 
$\Vert \beta^{\lambda}\cdot W \Vert_{BMO} \leq C$. Therefore $W_{s}^{\lambda}=W_{s}-\int_{0}^{s}\beta_{r}^{\lambda}dr$
is a Brownian motion under an equivalent probability measure 
$dP^{\lambda}=\mathcal{E}(\beta_{s}^{\lambda}\cdot W)_{0}^{T}dP$.\par
The process $Y_{t}-Y_{t}^{\lambda}$ can be written as 
\begin{equation}
Y_{t}-Y_{t}^{\lambda}=\int_{t}^{T}dF_{s}-\int_{t}^{T}(Z_{s}-Z_{s}^{\lambda})dW_{s}^{\lambda} \nonumber
\end{equation}
where the scalar process $F$ satisfies
\begin{equation}
dF_{s}=(f(s,Y_{s},Z_{s})-f(s,Y_{s}^{\lambda},Z_{s}))ds+dK_{s}-dK_{s}^{\lambda}.
\nonumber
\end{equation}
Then we have
\begin{equation}
\begin{aligned}
(Y_{t}-Y_{t}^{\lambda})dF_{t} &\leq C(Y_{t}-Y_{t}^{\lambda})^{2}dt+(Y_{t}-Y_{t}^{\lambda})dK_{t}-(Y_{t}-Y_{t}^{\lambda})dK_{t}^{\lambda} \\
&\leq C(Y_{t}-Y_{t}^{\lambda})^{2}dt+(Y_{t}-Y_{t}^{\lambda})dK_{t} \\
&= C(Y_{t}-Y_{t}^{\lambda})^{2}dt+(Y_{t}-S_{t})dK_{t}-(Y_{t}^{\lambda}-S_{t})dK_{t} \\
&\leq C(Y_{t}-Y_{t}^{\lambda})^{2}dt+(Y_{t}^{\lambda}-S_{t})^{-}dK_{t}.
\nonumber
\end{aligned}
\end{equation}
By setting
\begin{equation}
D,N,\eta=0,R_{t} = \int_{0}^{t}(Y_{s}^{\lambda}-S_{s})^{-}dK_{s},\Upsilon_{t}=C t,\nonumber
\end{equation}
we obtain from \cite[Proposition A.1]{GWsupply}
that
\begin{equation} \label{step1}
E^{\lambda} \left[ \sup\limits_{0\leq t\leq T} \vert Y_{t}-Y_{t}^{\lambda} \vert^{p} \right] + E^{\lambda} \left[ \left( \int_{0}^{T} \vert Z_{t}-Z_{t}^{\lambda} \vert^{2}dt \right)^{\frac{p}{2}} \right] \leq 
C_{p}E^{\lambda} \left[ \left( \int_{0}^{T} (Y_{t}^{\lambda}-S_{t})^{-}dK_{t} 
\right )^{\frac{p}{2}} \right] 
\end{equation}
for any $p \geq 2$. \\
\textbf{Step2} We prove that $E \left[ \left( \int_{0}^{T} (Y_{t}^{\lambda}-S_{t})^{-}dK_{t}  \right )^{\frac{p}{2}} \right] \leq \frac{C}{\lambda^{\frac{p}{2}}}$. From assumption $(A4)$ and \cite[Proposition 4.2 and Remark 4.3]{EKPPQ}, we have
\begin{equation*}
\left \{
\begin{aligned}
&dK_{t}\leq \kappa_{t}\mathbb{I}_{\{Y_{t}=S_{t}\}}dt,\\
&\kappa_{t}:=(f(t,Y_{t},Z_{t})+U_{t})^{-}.
\end{aligned}
\right.
\end{equation*}
From the fact that 
$(Y,Z\cdot W) \in \mathcal{S}^{\infty}(\mathbb{R})\times BMO$ and $U\in \mathcal{S}^{p}$, we have $\kappa \in \mathcal{S}^{p}$. Then 
\begin{equation*}
\begin{aligned}
E \left[ \left( \int_{0}^{T} (Y_{t}^{\lambda}-S_{t})^{-}dK_{t}  \right )^{\frac{p}{2}} \right] 
&\leq E \left[ \left( \sup\limits_{0\leq t\leq T} \kappa_{t} \int_{0}^{T} (Y_{t}^{\lambda}-S_{t})^{-}dt  \right )^{\frac{p}{2}} \right] \\
&\leq \sqrt{E \left[ \sup\limits_{0\leq t\leq T} \kappa_{t}^{p} \right]} 
\sqrt{E \left[ \left( \int_{0}^{T} (Y_{t}^{\lambda}-S_{t})^{-}dt \right)^{p} \right]}  \\
&\leq C\sqrt{E \left[ \left( \int_{0}^{T} (Y_{t}^{\lambda}-S_{t})^{-}dt \right)^{p} \right]}.
\end{aligned}
\end{equation*}
In addition, Lemma \ref{l2} and \ref{e2} indicate that 
\begin{equation*}
E \left[ \left( \lambda \int_{0}^{T} (Y_{s}^{\lambda}-S_{s})^{-}ds \right)^{p}
\right] \leq C.
\end{equation*} 
Then we deduce that
\begin{equation} \label{step2}
E \left[ \left( \int_{0}^{T} (Y_{t}^{\lambda}-S_{t})^{-}dK_{t}  \right )^{\frac{p}{2}} \right] \leq \frac{C}{\lambda^{\frac{p}{2}}}.
\end{equation}
\textbf{Step3} From Kazamaki \cite{KBMO}, there exist $p_{1}>1$ and $p_{2}>1$ such that 
\begin{equation} \label{step3}
E^{\lambda} \left[ \left( \frac{dP}{dP^{\lambda}} \right)^{p_{1}} \right]\leq C,
E^{\lambda} \left[ \left( \frac{dP^{\lambda}}{dP} \right)^{p_{2}} \right]\leq C.
\end{equation}
Therefore, from \eqref{step1}, \eqref{step2}, \eqref{step3} and H{\" o}lder's inequality, we deduce that for any $p \geq 2$,
\begin{align*}
&E \left[ \sup\limits_{0\leq t\leq T} \vert Y_{t}-Y_{t}^{\lambda} \vert^{p}  +  \left( \int_{0}^{T} \vert Z_{t}-Z_{t}^{\lambda} \vert^{2}dt \right)^{\frac{p}{2}} \right] \\
= & E^{\lambda} \left[ \left( \frac{dP}{dP^{\lambda}} \right)
\left( \sup\limits_{0\leq t\leq T} \vert Y_{t}-Y_{t}^{\lambda} \vert^{p}  +  \left( \int_{0}^{T} \vert Z_{t}-Z_{t}^{\lambda} \vert^{2}dt \right)^{\frac{p}{2}}  \right) \right] \\ 
\leq &  CE^{\lambda} \left[  \left( \frac{dP}{dP^{\lambda}} \right)^{p_{1}} \right]^{\frac{1}{p_{1}}}
  E^{\lambda} \left[ \left( \sup\limits_{0\leq t\leq T} \vert Y_{t}-Y_{t}^{\lambda} \vert^{pq_{1}}  +  \left( \int_{0}^{T} \vert Z_{t}-Z_{t}^{\lambda} \vert^{2}dt \right)^{\frac{pq_{1}}{2}} \right) \right]^{\frac{1}{q_{1}}} \\
\leq & C_{p}E^{\lambda} \left[ \left( \int_{0}^{T} (Y_{t}^{\lambda}-S_{t})^{-}dK_{t} 
\right )^{\frac{pq_{1}}{2}} \right]^{\frac{1}{q_{1}}} \\
\leq & CE \left[ \left( \frac{dP^{\lambda}}{dP} \right)^{p_{2}} \right]^{\frac{1}{q_{1}p_{2}}}
E \left[ \left( \int_{0}^{T} (Y_{t}^{\lambda}-S_{t})^{-}dK_{t} 
\right )^{\frac{pq_{1}q_{2}}{2}} \right]^{\frac{1}{q_{1}q_{2}}}\\
\leq & \left(\frac{C}{\lambda^{\frac{pq_{1}q_{2}}{2}}} \right)^{\frac{1}{q_{1}q_{2}}} \\
= & \frac{C}{\lambda^{\frac{p}{2}}}
\end{align*}
where $\frac{1}{p_{1}}+\frac{1}{q_{1}}=1,\frac{1}{p_{2}}+\frac{1}{q_{2}}=1$. Moreover, in \eqref{y-ylambda} we have
\begin{equation*}
K_{t}-K_{t}^{\lambda}=(Y_{0}-Y_{0}^{\lambda})-(Y_{t}-Y_{t}^{\lambda})-\int_{0}^{t} (f(s,Y_{s},Z_{s})-f(s,Y_{s}^{\lambda},Z_{s}))ds+\int_{0}^{t}(Z_{s}-Z_{s}^{\lambda}) dW_{s}^{\lambda},
\nonumber
\end{equation*}
which implies that
\begin{equation*}
E^{\lambda} \left[ \sup\limits_{0\leq t\leq T} \vert K_{t}-K_{t}^{\lambda} \vert^{p} \right] \leq C_{p} E^{\lambda} \left[ \sup\limits_{0\leq t\leq T} \vert Y_{t}-Y_{t}^{\lambda} \vert^{p}  +  \left( \int_{0}^{T} \vert Z_{t}-Z_{t}^{\lambda} \vert^{2}dt \right)^{\frac{p}{2}} \right].  \nonumber
\end{equation*}
Using the same BMO technique, we obtain that
\begin{equation*}
E \left[ \sup\limits_{0\leq t\leq T} \vert K_{t}-K_{t}^{\lambda} \vert^{p} \right] \leq \frac{C_{p}}{\lambda^{\frac{p}{2}}}.   \nonumber
\end{equation*}
\end{proof}

\section{Numerical approximation of reflected BSDEs with sub-quadratic generators} \label{label:Numerical approximation of RBSDEs with sub-quadratic generators}
In this section we consider the discretization error between reflected BSDE \eqref{e1} and its numerical solution using the estimation obtained from the previous section. Suppose the reflected BSDE satisfies assumptions $(A2)-(A4)$ and the following assumption:\par
$(A1^{'})$ The function $f:\Omega \times [0,T] \times \mathbb{R} \times \mathbb{R}^{d} \rightarrow \mathbb{R}$ satisfies that $f(\cdot,y,z)$ is adapted for each $y \in \mathbb{R}$ and $z \in \mathbb{R}^{d}$. It holds that
\begin{equation*}
\begin{aligned}
&\vert f(t,y,z) \vert \leq \overline{C}(1+ \vert y \vert + \vert z \vert^{2-\varepsilon}),\\
&\vert f(t,y_{1},z_{1})-f(t,y_{2},z_{2}) \vert \leq \overline{C} \vert y_{1}-y_{2} \vert 
+ \overline{C}(1+ \vert z_{1} \vert + \vert z_{2} \vert)\vert z_{1}-z_{2} \vert
\end{aligned} \nonumber
\end{equation*}
for $y,y_{1},y_{2} \in \mathbb{R}$, $z,z_{1},z_{2} \in \mathbb{R}^{d}$, positive constant $\overline{C}$ and $\varepsilon \in (0,2]$.\par
We firstly construct numerical scheme for penalized BSDE \eqref{e2} which satisfies $(A1^{'})$, $(A2)$, $(A3)$ through extending the Euler’s polygonal line method proposed by \cite{LT}. Define
\begin{equation} \label{num-PBSDE}
Y_{t}^{n,i,\lambda}=Y_{T_{n-i+1}}^{n,i-1,\lambda}+\int_{t}^{T_{n-i+1}} f^{\lambda} \left( s,
E^{\mathcal{F}_{s}} \left[ Y_{T_{n-i+1}}^{n,i-1,\lambda} \right],Z_{s}^{n,i,\lambda} \right)ds-\int_{t}^{T_{n-i+1}} 
Z_{s}^{n,i,\lambda} dW_{s}
\end{equation}
where $f^{\lambda}(t,y,z)=f(t,y,z)+\lambda(y-S_{t})^{-}$, $t\in[T_{n-i},T_{n-i+1}]$, $T_{i}:=\frac{i}{n}T$, $Y_{T}^{n,0,\lambda}=\xi$. The global solution for $t \in [0,T]$ has the form
\begin{equation*}
\begin{aligned}
Y_{t}^{n,\lambda}&:=\sum\limits_{i=1}^{n}Y_{t}^{n,i,\lambda} \cdot \mathbb{I}_{[T_{n-i},T_{n-i+1})}(t) + \xi \cdot \mathbb{I}_{\{T\}}(t),\\
Z_{t}^{n,\lambda}&:=\sum\limits_{i=1}^{n}Z_{t}^{n,i,\lambda} \cdot \mathbb{I}_{[T_{n-i},T_{n-i+1})}(t) + Z_{T}^{n,1,\lambda} \cdot \mathbb{I}_{\{T\}}(t),
\end{aligned}
\end{equation*}
which can be expressed as  the solution of such BSDE

\begin{equation} \label{e3}
Y_{t}^{n,\lambda}=\xi+\int_{t}^{T}f^{\lambda}\left(s,E^{\mathcal{F}_{s}}\left[Y_{\phi(n,s)}^{n,\lambda}\right],Z_{s}^{n,\lambda} \right)ds-\int_{t}^{T}Z_{s}^{n,\lambda}dW_{s},
\end{equation}
where $\phi(n,t):=\frac{[\frac{nt}{T}]+1}{n}T$ and  $\phi(n,T):=T$. \par


Setting
\begin{equation*}
\mathcal{S}^{\infty,i}(\mathbb{R}):=\mathcal{S}^{\infty}_{[T_{n-i},T_{n-i+1}]}(\mathbb{R}),\quad \mathcal{H}^{p,i}(\mathbb{R}^{d}):=\mathcal{H}^{p}_{[T_{n-i},T_{n-i+1}]}(\mathbb{R}^{d}),
\end{equation*}
we recall the following Lemma from \cite{LT}.
\begin{lemma} \label{li_lemma}
Let assumptions $(A1^{'}),~(A2)$ and $(A3)$ be satisfied. Then for $n \geq 1$ and $i=1,\dots,n$, BSDE \eqref{num-PBSDE} has a unique solution $(Y_{t}^{n,i,\lambda},Z_{t}^{n,i,\lambda}) \in \mathcal{S}^{\infty,i}(\mathbb{R}) \times \mathcal{H}^{2,i}(\mathbb{R}^{d})$ and \eqref{e3} has a unique solution $(Y_{t}^{n,\lambda},Z_{t}^{n,\lambda}) \in \mathcal{S}^{\infty}(\mathbb{R}) \times BMO$ and
\begin{equation}\label{e4}
\Vert Y^{n,\lambda}-Y^{\lambda} \Vert_{\mathcal{S}^{\infty}(R)}+\Vert Z^{n,\lambda}-Z^{\lambda} \Vert_{BMO} \leq \frac{C_{\lambda}}{n^{\frac{\varepsilon}{4}}}.
\end{equation} 
where $C_{\lambda}$ depends on $M_{\lambda}:=\sup \limits_{n\in \mathbb{N}} \left\{ \Vert Y^{n,\lambda} \Vert_{\mathcal{S}^{\infty}(R)} \right\}$ and $N_{\lambda}:=\sup \limits_{n\in \mathbb{N}} \left\{ \Vert Z^{n,\lambda}\cdot W \Vert_{BMO} \right\}$.
\end{lemma}
\begin{remark}
The uniform bound $M_{\lambda} \sim e^{\lambda}$ if we directly replace $f$ to $f_{\lambda}$ in the proof of \cite[Theorem 3.1]{LT}, which leads to the global discretization error has the form $\left( \frac{1}{lnn}\right)^{C_{\varepsilon}}$.
\end{remark}
We now provide a more delicate estimate for $M_{\lambda}$ if the following assumption is satisfied.\\
$(H)$: The generator $f$ is non-decreasing with respect to $y$, i.e., $\forall (t,z) \in [0,T]\times\mathbb{R}^{d}$ and $y_1,y_2\in\mathbb{R}$ with $y_1\leq y_2$, $f(t, y_1,z) \leq f(t, y_2,z)$.

\begin{lemma} \label{boundary_lemma}
Let assumptions $(A1^{'})$, $(A2)$, $(A3)$ and $(H)$ be satisfied, then for numerical scheme \eqref{e3} we have
\begin{equation*}
M_{\lambda} \leq C\lambda,\quad N_{\lambda}\leq C\lambda^{max\{2,\frac{2}{\epsilon}\}}.
\end{equation*}
\begin{proof}
For any $i=1,\dots,n$ and $t\in[T_{n-i},T_{n-i+1}]$, let $(Y_{t}^{n,i,0},Z_{t}^{n,i,0})$ be the unique solution of 
\begin{equation*}
Y_{t}^{n,i,0}=Y_{T_{n-i+1}}^{n,i-1,0}+\int_{t}^{T_{n-i+1}} f \left( s,
E^{\mathcal{F}_{s}} \left[ Y_{T_{n-i+1}}^{n,i-1,0} \right],Z_{s}^{n,i,0} \right)ds-\int_{t}^{T} 
Z_{s}^{n,i,0} dW_{s}
\end{equation*}
where $Y_{T}^{n,0,0}=\xi$. From  \cite[Lemma 2.1]{HT}, we have $ Y_{t}^{n,1,\lambda} \geq Y_{t}^{n,1,0}$ in $[T_{n-1},T_{n}]$, which implies that
\begin{equation*}
f\left( s,E^{\mathcal{F}_{s}} \left[ Y_{T_{n-1}}^{n,1,0} \right],z \right) \leq f\left( s,E^{\mathcal{F}_{s}} \left[ Y_{T_{n-1}}^{n,1,\lambda} \right],z \right) \leq f^{\lambda}\left( s,E^{\mathcal{F}_{s}} \left[ Y_{T_{n-1}}^{n,1,\lambda} \right],z \right).
\end{equation*}
Thus, we can deduce inductively that $Y_{t}^{n,2,\lambda} \geq Y_{t}^{n,2,0}$ in $[T_{n-2},T_{n-1}]$ and $Y_{t}^{n,\lambda} \geq Y_{t}^{n,0}$ in $[0,T]$. Note that $Y_{t}^{n,0}$ is independent of $\lambda$, and there exists 
\begin{equation*}
K:=\sup \limits_{n\in \mathbb{N}} \left\{ \Vert Y^{n,0} \Vert_{\mathcal{S}^{\infty}(R)} \right\} \leq C
\end{equation*}
from \cite{LT}, which implies that $Y_{t}^{n,\lambda} \geq -K$.\par
On the other hand, let $(\hat{Y}_{t}^{n,i,\lambda},\hat{Z}_{t}^{n,i,\lambda})$ be the unique solution of
\begin{equation*}
\hat{Y}_{t}^{n,i,\lambda}=\hat{Y}_{T_{n-i+1}}^{n,i-1,\lambda}+\int_{t}^{T_{n-i+1}} g^{\lambda} \left( s,
E^{\mathcal{F}_{s}} \left[ \hat{Y}_{T_{n-i+1}}^{n,i-1,\lambda} \right],\hat{Z}_{s}^{n,i,\lambda} \right)ds-\int_{t}^{T_{n-i+1}} 
\hat{Z}_{s}^{n,i,\lambda} dW_{s}
\end{equation*}
where $g^{\lambda}(s,y,z)=f(s,y,z)+\lambda(K+S_{s}^{+})$ and $\hat{Y}_{T}^{n,0,\lambda}=\xi$. The comparison theorem implies that $Y_{t}^{n,1,\lambda} \leq \hat{Y}_{t}^{n,1,\lambda}$. We also have
\begin{equation*}
\begin{aligned}
f^{\lambda}\left( s,E^{\mathcal{F}_{s}} \left[ Y_{T_{n-1}}^{n,1,\lambda} \right],z \right) 
&\leq f \left( t,E^{\mathcal{F}_{s}} \left[ Y_{T_{n-1}}^{n,1,\lambda} \right],z \right) + \lambda \left( E^{\mathcal{F}_{s}} \left[ Y_{T_{n-1}}^{n,1,\lambda} \right]^{-}+S_{s}^{+} \right) \\
&\leq g^{\lambda} \left( s,E^{\mathcal{F}_{s}} \left[ \hat{Y}_{T_{n-1}}^{n,1,\lambda} \right],z \right).
\end{aligned}
\end{equation*}
Then using the same technique we have $Y_{t}^{n,2,\lambda} \leq \hat{Y}_{t}^{n,2,\lambda}$ in $[T_{n-2},T_{n-1}]$ and $Y_{t}^{n,\lambda} \leq \hat{Y}_{t}^{n,\lambda}$ in $[0,T]$. 
Using the inequality $\vert z \vert^{2-\varepsilon} \leq 1+\vert z \vert^{2}$ we have
\begin{equation*}
\begin{aligned}
g^{\lambda} \left( s,
E^{\mathcal{F}_{s}} \left[ \hat{Y}_{T_{n-i+1}}^{n,i-1,\lambda} \right],z \right)
&\leq \overline{C} \left( 1+\Vert \hat{Y}^{n,i-1,\lambda} \Vert_{\mathcal{S}^{\infty,i-1}(\mathbb{R})} + \vert z \vert^{2-\varepsilon} \right) + \lambda \left( K+\Vert S^{+} \Vert_{\mathcal{S}^{\infty}(\mathbb{R})} \right)\\
&\leq \left( 2\overline{C}+\lambda K+\lambda \Vert S^{+} \Vert_{\mathcal{S}^{\infty}(\mathbb{R})} \right) + \overline{C} \Vert \hat{Y}^{n,i-1,\lambda} \Vert_{\mathcal{S}^{\infty,i-1}(\mathbb{R})} + \overline{C} \vert z \vert^{2} \\
&= \alpha_{\lambda}+\overline{C} \Vert \hat{Y}^{n,i-1,\lambda} \Vert_{\mathcal{S}^{\infty,i-1}(\mathbb{R})} + \overline{C} \vert z \vert^{2}.
\end{aligned}
\end{equation*}
Then from \cite[Lemma 2.1]{HT}, we have
\begin{equation*}
\begin{aligned}
\Vert \hat{Y}^{n,i,\lambda} \Vert_{\mathcal{S}^{\infty,i}(\mathbb{R})}
\leq &\alpha_{\lambda}\frac{T}{n} + \left( 1+\frac{\overline{C}T}{n} \right) \Vert \hat{Y}^{n,i-1,\lambda} \Vert_{\mathcal{S}^{\infty,i-1}(\mathbb{R})}\\
\leq &\alpha_{\lambda} \frac{T}{n} + \left( 1+\frac{\overline{C}T}{n} \right)
\left[  \alpha_{\lambda}\frac{T}{n} + \left( 1+\frac{\overline{C}T}{n} \right) \Vert \hat{Y}^{n,i-2,\lambda} \Vert_{\mathcal{S}^{\infty,i-2}(\mathbb{R})} \right]\\
= &\alpha_{\lambda} \frac{T}{n} \left[ 1+\left( 1+\frac{\overline{C}T}{n} \right) \right] + \left( 1+\frac{\overline{C}T}{n} \right)^{2}\Vert \hat{Y}^{n,i-2,\lambda} \Vert_{\mathcal{S}^{\infty,i-2}(\mathbb{R})}\\
\leq &\alpha_{\lambda} \frac{T}{n} \left[ 1+\left( 1+\frac{\overline{C}T}{n} \right)+\left( 1+\frac{\overline{C}T}{n} \right)^{2}+\cdots+ \left( 1+\frac{\overline{C}T}{n} \right)^{i-1}\right]+\left( 1+\frac{\overline{C}T}{n} \right)^{i} \Vert \xi \Vert_{\infty} \\
\leq &\frac{\alpha_{\lambda}}{\overline{C}}\left[ \left( 1+\frac{\overline{C}T}{n} \right)^{n}-1 \right]+\left( 1+\frac{\overline{C}T}{n} \right)^{n} \Vert \xi \Vert_{\infty} \\
\leq&\frac{\alpha_{\lambda}}{\overline{C}}\left( e^{\overline{C}T}-1 \right) +
e^{\overline{C}}\Vert \xi \Vert_{\infty},
\end{aligned}
\end{equation*}
which means that
\begin{equation*}
\sup \limits_{n\in \mathbb{N}} \left\{ \Vert \hat{Y}^{n,\lambda} \Vert_{\mathcal{S}^{\infty}(R)} \right\} \leq C\lambda
\end{equation*}
and $-K \leq M_{\lambda} \leq C\lambda.$\par
For any bounded stopping time $\tau$, applying It\^{o}'s formula to $\vert Y^{n,\lambda} \vert^{2}$ yields that
\begin{equation*}
\begin{aligned}
&\vert Y_{\tau}^{n,\lambda} \vert^{2}+E^{\mathcal{F}_{\tau}}\left[ \int_{\tau}^{T} \vert Z_{t}^{n,\lambda} \vert^{2}dt \right]\\
=&E^{\mathcal{F}_{\tau}}\left[ \vert \xi \vert^2 \right]+2E^{\mathcal{F}_{\tau}}\left[ \int_{\tau}^{T} Y_{t}^{n,\lambda} f^{\lambda}\left(t,E^{\mathcal{F}_{t}}\left[Y_{\phi(n,t)}^{n,\lambda}\right],Z_{t}^{n,\lambda} \right)dt \right]\\
\leq& \Vert \xi \Vert_{\infty}+2E^{\mathcal{F}_{\tau}}\left[ \int_{\tau}^{T} \vert Y_{t}^{n,\lambda} \vert \left( \overline{C}+ \overline{C}M_{\lambda}+\overline{C}\vert Z_{t}^{n,\lambda} \vert^{2-\varepsilon} +\lambda K +\lambda \Vert S^{+} \Vert_{\mathcal{S}^{\infty}(\mathbb{R})} \right)dt \right]\\
=&\Vert \xi \Vert_{\infty}+2E^{\mathcal{F}_{\tau}}\left[ \int_{\tau}^{T} \vert Y_{t}^{n,\lambda} \vert \left( C+ C\lambda +C\vert Z_{t}^{n,\lambda} \vert^{2-\varepsilon}  \right)dt \right].
\end{aligned}
\end{equation*}
Using Young's equality we have $2C\vert Y_{t}^{n,\lambda} \vert\vert Z_{t}^{n,\lambda} \vert^{2-\varepsilon}\leq \frac{\varepsilon}{2} \left( 2C \vert Y_{t}^{n,\lambda} \vert \right)^{\frac{2}{\varepsilon}}+\frac{2-\varepsilon}{2}\vert Z_{t}^{n,\lambda} \vert^{2}$ and thus
\begin{equation*}
E^{\mathcal{F}_{\tau}}\left[ \int_{\tau}^{T} \vert Z_{t}^{n,\lambda} \vert^{2}dt \right]\leq C+C\lambda+C\lambda^{2}+C\lambda^{\frac{2}{\epsilon}}.
\end{equation*}
Therefore we have 
\begin{equation*}
\sup \limits_{n\in \mathbb{N}} \left\{ \Vert Z^{n,\lambda}\cdot W \Vert_{BMO} \right\} \leq C\lambda^{max\{2,\frac{2}{\epsilon}\}}.
\end{equation*}
\end{proof}
\end{lemma}\par
Now we are ready for the following theorem.
\begin{theorem} \label{th2}
Let assumptions $(A1^{'})$, $(A2)-(A4)$ and $(H)$ be satisfied, then for $p \geq 2$,
\begin{equation}
\Vert Y^{n,\lambda}-Y \Vert_{\mathcal{S}^{p}(\mathbb{R})} \leq C_p\sqrt{\frac{\lambda^{max\{8-4\varepsilon,\frac{8}{\epsilon}-4\}}}{n^{\varepsilon+1}}+\frac{\lambda^{2}}{n}+\frac{1}{\lambda}}.
\end{equation}
\begin{proof}
For any $i=1,\dots,n$, applying It\^{o}'s formula to $e^{2\overline{C}t}\vert Y^{n,i,\lambda}_{t}-Y^{\lambda}_{t} \vert^{2}$ in the interval $[T_{n-i},T_{n-i+1}]$ yields that
\begin{equation*}
\begin{aligned}
&e^{2\overline{C}t}\vert Y^{n,i,\lambda}_{t}-Y^{\lambda}_{t} \vert^{2}+\int_{t}^{T_{n-i+1}}2e^{2\overline{C}s}(Y^{n,i,\lambda}_{s}-Y^{\lambda}_{s})(Z^{n,i,\lambda}_{s}-Z^{\lambda}_{s})dW_{s} \\
&+\int_{t}^{T_{n-i+1}}e^{2\overline{C}s}\vert Z^{n,i,\lambda}_{s}-Z^{\lambda}_{s} \vert^{2}ds \\
=&e^{2\overline{C}T_{n-i+1}}\vert Y^{n,i,\lambda}_{T_{n-i+1}}-Y^{\lambda}_{T_{n-i+1}} \vert^{2} - \int_{t}^{T_{n-i+1}}2\overline{C}e^{2\overline{C}s}\vert Y^{n,i,\lambda}_{s}-Y^{\lambda}_{s} \vert^{2}ds \\
&+\int_{t}^{T_{n-i+1}}2e^{2\overline{C}s}(Y^{n,i,\lambda}_{s}-Y^{\lambda}_{s})\left[ f^{\lambda}(s,E^{\mathcal{F}_{s}} [ Y_{T_{n-i+1}}^{n,i-1,\lambda} ],Z_{s}^{n,i,\lambda})-f^{\lambda}(s,Y_{s}^{\lambda},Z_{s}^{\lambda}) \right]ds\\
\leq&e^{2\overline{C}T_{n-i+1}}\vert Y^{n,i,\lambda}_{T_{n-i+1}}-Y^{\lambda}_{T_{n-i+1}} \vert^{2} - \int_{t}^{T_{n-i+1}}2\overline{C}e^{2\overline{C}s}\vert Y^{n,i,\lambda}_{s}-Y^{\lambda}_{s} \vert^{2}ds \\
&+\int_{t}^{T_{n-i+1}}2e^{2\overline{C}s}(Y^{n,i,\lambda}_{s}-Y^{\lambda}_{s})\left[ f^{\lambda}(s,E^{\mathcal{F}_{s}} [ Y_{T_{n-i+1}}^{n,i-1,\lambda} ],Z_{s}^{n,i,\lambda})-f^{\lambda}(s,Y_{s}^{\lambda},Z_{s}^{n,i,\lambda}) \right]ds\\
&+\int_{t}^{T_{n-i+1}}2e^{2\overline{C}s}(Y^{n,i,\lambda}_{s}-Y^{\lambda}_{s})\left[ f^{\lambda}(s,Y_{s}^{\lambda},Z_{s}^{n,i,\lambda})-f^{\lambda}(s,Y_{s}^{\lambda},Z_{s}^{\lambda}) \right]ds.
\end{aligned}
\end{equation*}
We write 
\begin{equation*}
f^{\lambda}(t,Y_{t}^{\lambda},Z_{t}^{n,\lambda})-f^{\lambda}(t,Y_{t}^{\lambda},Z_{t}^{\lambda})=\beta_{t}^{n,\lambda}(Z^{n,\lambda}_{t}-Z^{\lambda}_{t}), t \in [0,T].
\end{equation*}
Then we have $\vert \beta_{t}^{n,\lambda} \vert \leq \overline{C}(1+\vert Z^{n,\lambda}_{t} \vert+\vert Z^{\lambda}_{t} \vert)$ and  
\begin{equation*}
\sup \limits_{n\in \mathbb{N}} \left\{ \Vert \beta^{n,\lambda}\cdot W \Vert_{BMO} \right\} \leq C_{\lambda}
\end{equation*}
from Lemma \ref{l2} and \ref{boundary_lemma}. We define $W_{t}^{n,\lambda}=W_{t}-\int_{0}^{s}\beta_{s}^{n,\lambda}ds$ under an equivalent probability measure $dP^{n,\lambda}=\mathcal{E}(\beta_{s}^{n,\lambda}\cdot W)_{0}^{T}dP$. The inequality can be written as 
\begin{equation*}
\begin{aligned}
&e^{2\overline{C}t}\vert Y^{n,i,\lambda}_{t}-Y^{\lambda}_{t} \vert^{2}+\int_{t}^{T_{n-i+1}}2e^{2\overline{C}s}(Y^{n,i,\lambda}_{s}-Y^{\lambda}_{s})(Z^{n,i,\lambda}_{s}-Z^{\lambda}_{s})dW_{s}^{n,\lambda} \\
&+\int_{t}^{T_{n-i+1}}e^{2\overline{C}s}\vert Z^{n,i,\lambda}_{s}-Z^{\lambda}_{s} \vert^{2}ds \\
\leq & e^{2\overline{C}T_{n-i+1}}\vert Y^{n,i,\lambda}_{T_{n-i+1}}-Y^{\lambda}_{T_{n-i+1}} \vert^{2} - \int_{t}^{T_{n-i+1}}2\overline{C}e^{2\overline{C}s}\vert Y^{n,i,\lambda}_{s}-Y^{\lambda}_{s} \vert^{2}ds \\
&+\int_{t}^{T_{n-i+1}}2e^{2\overline{C}s}(Y^{n,i,\lambda}_{s}-Y^{\lambda}_{s})\left[ f^{\lambda}(s,E^{\mathcal{F}_{s}} [ Y_{T_{n-i+1}}^{n,i-1,\lambda} ],Z_{s}^{n,i,\lambda})-f^{\lambda}(s,Y_{s}^{\lambda},Z_{s}^{n,i,\lambda}) \right]ds\\
\leq & e^{2\overline{C}T_{n-i+1}}\vert Y^{n,i,\lambda}_{T_{n-i+1}}-Y^{\lambda}_{T_{n-i+1}} \vert^{2}+ \int_{t}^{T_{n-i+1}}2e^{2\overline{C}s} \vert Y^{n,i,\lambda}_{s}-Y^{\lambda}_{s} \vert \vert E^{\mathcal{F}_{s}} [ Y_{T_{n-i+1}}^{n,i-1,\lambda} ]-Y^{n,i,\lambda}_{s}\vert ds \\
&+\int_{t}^{T_{n-i+1}}2e^{2\overline{C}s} \vert Y^{n,i,\lambda}_{s}-Y^{\lambda}_{s} \vert \vert \lambda (E^{\mathcal{F}_{s}} [ Y_{T_{n-i+1}}^{n,i-1,\lambda} ]-S_{s})^{-}-\lambda(Y^{\lambda}_{s}-S_{s})^{-}\vert ds
\end{aligned}
\end{equation*}
We define
\begin{equation*}
\begin{aligned}
&p_{t}^{n,i,\lambda}=E^{\mathcal{F}_{t}} [ Y_{T_{n-i+1}}^{n,i-1,\lambda} ]-Y^{n,i,\lambda}_{t}, \\
&q_{t}^{n,i,\lambda}=\lambda (E^{\mathcal{F}_{t}} [ Y_{T_{n-i+1}}^{n,i-1,\lambda} ]-S_{t})^{-}-\lambda(Y^{\lambda}_{t}-S_{t})^{-}.
\end{aligned}
\end{equation*}
Note that $p_{t}^{n,i,\lambda}$ and $q_{t}^{n,i,\lambda}$ are bounded by $P^{n,\lambda}$ and $Q^{n,\lambda}$ respectively from Lemma \ref{lA1} , then we have
\begin{equation*}
\begin{aligned}
&\vert Y^{n,i,\lambda}_{t}-Y^{\lambda}_{t} \vert^{2}+\int_{t}^{T_{n-i+1}}2e^{2\overline{C}(s-t)}(Y^{n,i,\lambda}_{s}-Y^{\lambda}_{s})(Z^{n,i,\lambda}_{s}-Z^{\lambda}_{s})dW_{s}^{n,\lambda} \\
&+\int_{t}^{T_{n-i+1}}e^{2\overline{C}(s-t)}\vert Z^{n,i,\lambda}_{s}-Z^{\lambda}_{s} \vert^{2}ds \\
\leq & e^{2\overline{C}\frac{T}{n}}\vert Y^{n,i,\lambda}_{T_{n-i+1}}-Y^{\lambda}_{T_{n-i+1}} \vert^{2}+ P^{n,\lambda}\int_{t}^{T_{n-i+1}}2e^{2\overline{C}(s-t)} \vert Y^{n,i,\lambda}_{s}-Y^{\lambda}_{s} \vert ds \\
&+Q^{n,\lambda}\int_{t}^{T_{n-i+1}}2e^{2\overline{C}(s-t)} \vert Y^{n,i,\lambda}_{s}-Y^{\lambda}_{s} \vert ds \\
\leq & e^{2\overline{C}\frac{T}{n}}\vert Y^{n,i,\lambda}_{T_{n-i+1}}-Y^{\lambda}_{T_{n-i+1}} \vert^{2} + (e^{2\overline{C}\frac{T}{n}}P^{n,\lambda})^{2}\frac{2T^{2}}{n}+(e^{2\overline{C}\frac{T}{n}}Q^{n,\lambda})^{2}\frac{2T^{2}}{n}\\
&+ \frac{n}{T^{2}}\left(\int_{t}^{T_{n-i+1}}\vert Y^{n,i,\lambda}_{s}-Y^{\lambda}_{s} \vert ds \right)^{2}
\end{aligned}
\end{equation*}
Taking conditional expectation with respect to $\mathcal{F}_{t}$ and $E^{n,\lambda}$, we have
\begin{align*}
\Vert Y^{n,i,\lambda}_{t}-Y^{\lambda}_{t} \Vert_{\mathcal{S}^{\infty,i}(\mathbb{R})}^{2} 
\leq&  \frac{e^{\frac{C}{n}}}{1-\frac{1}{n}}\Vert Y^{n,i-1,\lambda}_{t}-Y^{\lambda}_{t} \Vert_{\mathcal{S}^{\infty,i-1}(\mathbb{R})}^{2}+\frac{C}{n} \left( (P^{n,\lambda})^{2}+(Q^{n,\lambda})^{2} \right)\\
\leq&  \left( \frac{e^{\frac{C}{n}}}{1-\frac{1}{n}} \right)^{k}\Vert Y^{n,i-k,\lambda}_{t}-Y^{\lambda}_{t} \Vert_{\mathcal{S}^{\infty,i-k}(\mathbb{R})}^{2}\\
&+\frac{C}{n}\left[ 1+\frac{e^{\frac{C}{n}}}{1-\frac{1}{n}}+\cdots+\left(\frac{e^{\frac{C}{n}}}{1-\frac{1}{n}} \right)^{k-1} \right]\left( (P^{n,\lambda})^{2}+(Q^{n,\lambda})^{2} \right)\\
\leq&  \frac{C}{n}\left[ \left( \frac{e^{\frac{C}{n}}}{1-\frac{1}{n}} \right)^{n}-1 \right]\left( (P^{n,\lambda})^{2}+(Q^{n,\lambda})^{2} \right) \\
=&  \frac{C}{n} \left( (P^{n,\lambda})^{2}+(Q^{n,\lambda})^{2} \right)
\end{align*}
Thus we have 
\begin{equation*}
\Vert Y^{n,\lambda}_{t}-Y^{\lambda}_{t} \Vert_{\mathcal{S}^{\infty}(\mathbb{R})}^{2} \leq C(\frac{\lambda^{max\{8-4\varepsilon,\frac{8}{\epsilon}-4\}}}{n^{\varepsilon+1}}+\frac{\lambda^{2}}{n})
\end{equation*}
and 
\begin{equation*}
\Vert Y^{n,\lambda}-Y \Vert_{\mathcal{S}^{p}(\mathbb{R})} \leq C_p\sqrt{(\frac{\lambda^{max\{8-4\varepsilon,\frac{8}{\epsilon}-4\}}}{n^{\varepsilon+1}}+\frac{\lambda^{2}}{n}+\frac{1}{\lambda})}
\end{equation*}
from Theorem \ref{th1}.
\end{proof}
\end{theorem}

\appendix
\section{Appendix}
In this appendix, we provide some estimates for the proof of Theorem \ref{th2}.
\begin{lemma}\label{lA1}
Under the assumptions of Theorem \ref{th2}, define
\begin{equation*}
\begin{aligned}
&p_{t}^{n,\lambda}=E^{\mathcal{F}_{t}}[Y_{\phi(n,t)}^{n,\lambda}]-Y^{n,\lambda}_{t}, \\
&q_{t}^{n,\lambda}=\lambda (E^{\mathcal{F}_{t}}[Y_{\phi(n,t)}^{n,\lambda}]-S_{t})^{-}-\lambda(Y^{\lambda}_{t}-S_{t})^{-}.
\end{aligned}
\end{equation*}
Then $\vert p_{t}^{n,\lambda} \vert \leq C(1+\lambda)\frac{1}{n}+C\lambda^{max\{4-2\varepsilon,\frac{4}{\epsilon}-2\}}\left(\frac{1}{n} \right)^{\frac{\varepsilon}{2}}:=P^{n,\lambda}$, $\vert q_{t}^{n,\lambda} \vert \leq C\lambda :=Q^{n,\lambda}$.
\begin{proof}
From \eqref{e3} we have
\begin{equation*}
p_{t}^{n,\lambda}=-E^{\mathcal{F}_{t}} \left[ \int_{t}^{\phi(n,t)}  f^{\lambda} \left( s,E^{\mathcal{F}_{s}}\left[ Y_{\phi(n,s)}^{n,\lambda} \right],Z_{s}^{n,\lambda} \right)  \right],t\in[0,T].
\end{equation*}
Then 
\begin{equation*}
\begin{aligned}
\vert p_{t}^{n,\lambda} \vert &\leq E^{\mathcal{F}_{t}} \left[ \int_{t}^{\phi(n,t)} \vert f^{\lambda} \left( s,E^{\mathcal{F}_{s}}\left[ Y_{\phi(n,s)}^{n,\lambda} \right],Z_{s}^{n,\lambda} \right) \vert \right]\\
&\leq C(1+M_{\lambda}+\lambda)\frac{T}{n}+CE^{\mathcal{F}_{t}} \left[ \int_{t}^{\phi(n,t)}\vert Z_{s}^{n,\lambda} \vert^{2-\varepsilon}ds \right]\\
&\leq C(1+M_{\lambda}+\lambda)\frac{T}{n}+CE^{\mathcal{F}_{t}} \left[ \left( \int_{t}^{\phi(n,t)}\vert Z_{s}^{n,\lambda} \vert^{2}ds \right)^{1-\frac{\varepsilon}{2}} \left( \phi(n,t)-t \right)^{\frac{\varepsilon}{2}} \right]\\
&\leq C(1+M_{\lambda}+\lambda)\frac{T}{n}+CN_{\lambda}^{2-\varepsilon}\left(\frac{T}{n} \right)^{\frac{\varepsilon}{2}}\\
&=C(1+\lambda)\frac{T}{n}+C\lambda^{max\{4-2\varepsilon,\frac{4}{\epsilon}-2\}}\left(\frac{T}{n} \right)^{\frac{\varepsilon}{2}}.
\end{aligned}
\end{equation*}
On the other hand, 
\begin{equation*}
\begin{aligned}
\vert q_{t}^{n,\lambda} \vert &\leq C\lambda \left[ (E^{\mathcal{F}_{t}}[Y_{\phi(n,t)}^{n,\lambda}])^{-}+\Vert S^{+} \Vert_{\mathcal{S}^{\infty}(R)}+\Vert Y^{\lambda} \Vert_{\mathcal{S}^{\infty}(R)} \right]\\
&\leq C\lambda 
\end{aligned}
\end{equation*}
from Lemma \ref{l2} and \ref{boundary_lemma}.
\end{proof}
\end{lemma}

\end{document}